\documentclass{amsart}

\usepackage{amssymb}
\usepackage{amsthm}
\newtheorem{theorem}{Theorem}[section]

\newtheorem{lemma}[theorem]{Lemma}

\newtheorem*{conjecture}{Huppert's Conjecture}
\newcommand{\Z}{\mathbb{Z}}

\begin{document}


\title[Simple Ree Groups]{\bf The Simple Ree groups ${}^2F_4(q^2)$ are determined by the set of their character degrees}

\author{Hung P. Tong-Viet}
\email{Tong-Viet@ukzn.ac.za}
\address{School of Mathematical Sciences,
University of KwaZulu-Natal\\
Pietermaritzburg 3209, South Africa}

\date{\today}
\begin{abstract} Let $G$ be a finite group.  Let ${\rm{cd}}(G)$ be
the set of all complex irreducible character degrees of $G.$ In this paper, we will show that if
${\rm{cd}}(G)={\rm{cd}}(H),$ where $H$ is the simple Ree group ${}^2F_4(q^2),q^2\geq 8,$ then
$G\cong H\times A,$ where $A$ is an abelian group. This verifies Huppert's Conjecture for the
simple Ree groups ${}^2F_4(q^2)$ when $q^2\geq 8.$

\end{abstract}

\thanks{Support from the University of KwaZulu-Natal is acknowledged}
\maketitle

\section{Introduction and Notation}
All groups considered are finite and all characters are complex characters. For a group $G$ we
denote by $\textrm{Irr}(G)$  the set of all irreducible characters
of $G$ and let ${\rm{cd}}(G)=\{\chi(1)\:|\:\chi\in \textrm{Irr}(G)\}$ be the set of all character
degrees of $G.$

Huppert proposed the following conjecture in the late $1990s.$

\begin{conjecture}\label{HC} Let $G$ be a finite group and let $H$ be a nonabelian simple
group. If ${\rm{cd}}(G)= {\rm{cd}}(H),$ then $G\cong H\times A,$ where $A$ is abelian.
\end{conjecture}
Huppert verified this conjecture for $L_2(q)$ and $Sz(q^2)$ in \cite{Hupp} and several small
groups.   Recently, T. Wakefield verified this conjecture for some families of simple groups of Lie
type of Lie rank $2$ (see \cite{Wake}). The proof is based on verifying the following $5$ steps outlined
in \cite{Hupp}, which we call  {\em Huppert's Method}.

{\bf Step $1.$} Show $G'=G''.$ It follows that if $G'/M$  is a chief factor of $G,$ then $G'/M\cong
S^k,$ where $S$ is a nonabelian simple group and $k\geq 1.$

{\bf Step $2.$} Show $G'/M\cong H.$

{\bf Step $3.$} If $\theta\in \textrm{Irr}(M)$ and $\theta(1)=1,$ then $\theta$ is $G'$-invariant,
which implies  $[M,G']=M'.$

{\bf Step $4.$} Show $M=1,$  which implies $G'\cong H.$

{\bf Step $5.$} Show $G=G'\times C_G(G').$ As $G/G'\cong C_G(G')$ is abelian and $G'\cong H,$
Huppert's Conjecture follows.

In this paper, we will verify this conjecture for the simple exceptional group of Lie type
${}^2F_4(q^2),$ where $q^2=2^{2m+1},m\geq 1.$ This family of nonabelian simple groups was
discovered by Rimhak Ree in $1961$ and so called the {\em simple Ree groups.} We note that when
$m=0,$ the group ${}^2F_4(2)$ is not simple but its derived subgroup ${}^2F_4(2)'$ is simple. This
group is called the {\em Tits group}. In his preprint, Huppert already verified the conjecture for
this group and so we only need to consider the case when $m\geq 1.$ The irreducible characters of
${}^2F_4(q^2)$ were computed by G. Malle \cite{Malle90} and CHEVIE \cite{Chevie} and their maximal
subgroups were classified by Malle in \cite{Malle91}.
\begin{theorem}\label{main} Let $G$ be a finite
group and let ${H}$ be the simple Ree group of type ${}^2F_4(q^2),q^2=2^{2m+1},m\geq 1.$ If
${\rm{cd}}(G)= {\rm{cd}}({H}),$ then $G\cong H\times A,$ where $A$ is abelian.
\end{theorem}

Huppert's Method described above was improved by T. Wakefield in \cite{Wake}, especially
for Step $2.$ In this paper, we introduce the notion of an isolated character and use it to
simplify the proof of Step $1.$ For the definition of isolated characters, see the discussion right
after the proof of Lemma \ref{lem3}. The isolated character behaves like the Steinberg character of
the simple groups of Lie type and in fact this is an example of an isolated character (see Lemma
\ref{lem4}). Now by Lemma \ref{lem3} we can verify Step $1$ provided that we know several isolated
character degrees instead of all character degrees. This could be used to verify Step $1$ for all
simple groups of Lie type. In order to verify Step $3,$ we rely heavily on the criterion for the
character extension using Schur multiplier (see \cite[Theorem $11.7$]{Isaacs}) and a result of R.
Higgs on the fixed prime power projective character degrees (see \cite[Theorem \textrm{B}]{Higgs}).
Using the same method, one can verify Step $3$ for other simple groups of Lie type. In general, we
need to know all maximal subgroups of the simple group $H$ whose indices divide some character
degrees of $H$ and also the character degrees and the Schur multipliers of the nonabelian
composition factors involved in those maximal subgroups. This is in fact the most difficult step of
 Huppert's Method. Finally, in order to verify Step $5,$ we need to show that the character
degree sets of a simple group $H$ and any  almost simple group with socle $H$
 are different.

If $n$ is an integer then we denote by $\pi(n)$ the set of all prime divisors of $n.$ If $G$ is a
group, we will write $\pi(G)$ instead of $\pi(|G|)$ to denote the set of all prime divisors of the
order of $G.$ Let $\rho(G)=\cup_{\chi\in \textrm{Irr}(G)}\pi(\chi(1))$ be the set of all primes
which divide some irreducible character degrees of $G.$ If $N\unlhd G$ and $\theta\in
\textrm{Irr}(N),$ then the inertia group of $\theta$ in $G$ is denoted by $I_G(\theta).$ Finally, the set of
all irreducible constituents of $\theta^G$ is denoted by $\textrm{Irr}(G|\theta).$ Other notation is standard.
\section{Preliminaries}
In this section, we present some results that we will need for the proof of Huppert's
Conjecture.

\begin{lemma}\emph{(\cite[Lemma $2$]{Hupp}).}\label{lem1} Suppose $N\unlhd G$ and $\chi\in {\rm{Irr}}(G).$

$(a)$ If $\chi_N=\theta_1+\theta_2+\cdots+\theta_k$ with $\theta_i\in {\rm{Irr}}(N),$ then $k$
divides $|G/N|.$ In particular, if $\chi(1)$ is prime to $|G/N|$ then $\chi_N\in {\rm{Irr}}(N).$

$(b)$ \emph{(Gallagher's Theorem)} If $\chi_N\in {\rm{Irr}}(N),$ then $\chi\psi\in {\rm{Irr}}(G)$
for every $\psi\in {\rm{Irr}}(G/N).$

\end{lemma}

\begin{lemma}\emph{(\cite[Lemma $3$]{Hupp}).}\label{lem2} Suppose $N\unlhd
G$ and $\theta\in {\rm{Irr}}(N).$ Let $I=I_G(\theta).$

$(a)$ If $\theta^I=\sum_{i=1}^k\varphi_i$ with $\varphi_i\in {\rm{Irr}}(I),$ then $\varphi_i^G\in
{\rm{Irr}}(G).$ In particular, $\varphi_i(1)|G:I|\in {\rm{cd}}(G).$

$(b)$ If $\theta$ extends to $\psi\in {\rm{Irr}}( I),$ then $(\psi\tau )^G\in {\rm{Irr}}(G)$ for
all $\tau\in {\rm{Irr}}(I/N).$ In particular, $\theta(1)\tau(1)|G:I|\in {\rm{cd}}(G).$

$(c)$ If $\rho \in {\rm{Irr}}( I)$ such that $\rho_N=e\theta,$ then $\rho=\theta_0\tau_0,$ where
$\theta_0$ is a character of an irreducible projective representation of $I$ of degree $\theta(1)$
while $\tau_0$ is the character of an irreducible projective representation of $I/N$ of degree $e.$
\end{lemma}

The following lemma will be used to verify Step $1.$ All these statements but the last one appear
in \cite[Lemma $4$]{Hupp}. We will give a proof for completeness.
\begin{lemma}\label{lem3} Let $G/N$ be a solvable factor group of
$G,$ minimal with respect to being nonabelian. Then two cases can occur.

$(a)$ $G/N$ is an $r$-group for some prime $r.$ Hence there exists $\psi\in {\rm{Irr}}(G/N)$ such
that $\psi(1)=r^b>1.$ If $\chi\in {\rm{Irr}}(G)$ and $r\nmid \chi(1),$ then $\chi\tau\in
{\rm{Irr}}(G)$ for all $\tau\in {\rm{Irr}}(G/N).$

$(b)$ $G/N$ is a Frobenius group with an elementary abelian Frobenius kernel $F/N.$ Then
$f=|G:F|\in {\rm{cd}}(G)$ and $|F/N|=r^a$ for some prime $r,$ and $F/N$ is an irreducible module
for the cyclic group $G/F,$ hence $a$ is the smallest integer such that $r^a\equiv 1 (\mbox{mod
$f$}).$ If $\psi\in {\rm{Irr}}(F)$ then either $f\psi(1)\in {\rm{cd}}(G)$ or $r^a$ divides
$\psi(1)^2.$ In the latter case, $r$ divides $\psi(1).$

$(1)$ If no proper multiple of $f$ is in ${\rm{cd}}(G),$ then $\chi(1)$ divides $f$ for all
$\chi\in {\rm{Irr}}(G)$ such that $r\nmid \chi(1),$ and if  $\chi\in {\rm{Irr}}(G)$ such that $
\chi(1)\nmid f,$ then $r^a\mid \chi(1)^2.$

$(2)$ If $\chi\in {\rm{Irr}}(G)$ such that no proper multiple of $\chi(1)$ is in ${\rm{cd}}(G),$
then either $f$ divides $\chi(1)$ or $r^a$ divides $\chi(1)^2.$ Moreover if $\chi(1)$ is divisible
by no nontrivial proper character degree in $G,$ then $f=\chi(1)$ or $r^a\mid \chi(1)^2.$
\end{lemma}

\begin{proof}  Statements $(a)$ and $(b)$ follow from \cite[Lemma $2.3$]{Isaacs} and
\cite[Theorem $12.4$]{Isaacs}. Suppose that $G/N$ is a Frobenius group.

Now assume that no proper multiple of $f$ is in ${\rm{cd}}(G),$ and let $\chi\in {\rm{Irr}}(G).$
Let $\psi$ be an irreducible constituent of $\chi_F.$ By \cite[Lemma $6.8$]{Isaacs}, we have that
$\chi(1)=k\psi(1)$ and by \cite[Corollary $11.29$]{Isaacs} we obtain $k\mid f=|G:F|.$ By $(b)$, we
have that either $f\psi(1)\in {\rm{cd}}(G)$ or $r^a\mid \psi(1)^2.$ Suppose $r\nmid \chi(1).$ Then
$r\nmid \psi(1)$ so that $f\psi(1)=f\chi(1)/k\in \textrm{cd}(G).$ As no proper multiple of $f$ is a
character degree of $G,$ we deduce that $f\chi(1)/k=f$ so that $\chi(1)=k\mid f.$ Now assume
$\chi(1)\nmid f.$ Then  $r\mid \chi(1).$ Since $r\nmid f,$ we deduce that $r\nmid k,$ hence $r\mid
\psi(1)$ so that $f\psi(1)>f.$ Thus $f\psi(1)$ is not a character degree of $G$ and so $r^a\mid
\psi(1)^2.$ As $\psi(1)\mid \chi(1),$ $(1)$ follows. The proof of $(2)$ is exactly the same.

Suppose that $\chi\in \textrm{Irr}(G)$ such that no proper multiple of $\chi(1)$ is in
${\rm{cd}}(G).$ Let $\psi\in \textrm{Irr}(F)$ be an irreducible constituent of $\chi_F.$ As above,
we have that $\chi(1)=k\psi(1),$ $k\mid f$ and either $f\psi(1)\in {\rm{cd}}(G)$ or $r^a\mid
\psi(1)^2.$ If the latter case holds then we are done since $\psi(1)\mid \chi(1).$ Now assume
$f\psi(1)\in {\rm{cd}}(G).$ Observe that $\psi(1)=\chi(1)/k$ so that $\psi(1)f=f\chi(1)/k\in
{\rm{cd}}(G),$ where $f\chi(1)/k$ is a multiple of $\chi(1)$ since $k\mid f.$ As no proper multiple
of $\chi(1)$ belongs to ${\rm{cd}}(G),$ it follows that $f\chi(1)/k=\chi(1),$ which implies that
$f=k.$ Since $k$ divides $\chi(1),$ we deduce that $f\mid \chi(1).$ The remaining statement is
obvious. The proof is now complete.
\end{proof}

Let $\chi\in \textrm{Irr}(G).$ We say that $\chi$ is \emph{isolated} in $G$ if $\chi(1)$ is
divisible by no proper nontrivial character degree of $G,$ and no proper multiple of $\chi(1)$ is a
character degree of $G.$ In this situation, we also say that $\chi(1)$ is an \emph{isolated degree}
of $G.$ Recall that for $\chi\in \textrm{Irr}(G),$ $\chi$ is said to be of \emph{p-defect zero} for
some prime $p$ if $|G|/\chi(1)$ is coprime to $p.$
\begin{lemma}\label{lem4} If $S$ is a simple group of Lie type in characteristic $p$ with $S\neq {}^2F_4(2)',$
then the Steinberg character of $S$ of degree $|S|_p$ is an isolated character of $S.$
\end{lemma}

\begin{proof}  The existence of the Steinberg character of $S,$ denoted by $St_S,$ is well known.
From \cite[Theorem $1.1$]{Malle}, no proper nontrivial divisor of $St_S(1)$ is in ${\rm{cd}}(S).$
As $St_S$ is the only character of $p$-defect zero (see \cite[Theorem $4$]{Curtis}), no proper
multiple of $St_S(1)$ is in ${\rm{cd}}(S).$ This completes the proof.
\end{proof}

The next two lemmas will be used to verify Steps $2$ and $4.$ The first lemma appears in
\cite[Theorems $2,3,4$]{Bia}.

\begin{lemma}\label{lem5} If $S$ is
a nonabelian simple group, then there exists a nontrivial irreducible character $\theta$ of $S$
that extends to ${\rm{Aut}}(S).$ Moreover the following hold:

$(i)$ If $S$ is an alternating group of degree at least $7,$ then $S$ has two consecutive
characters of degrees $n(n-3)/2$ and $(n-1)(n-2)/2$ that both extend to ${\rm{Aut}}(S).$

$(ii)$ If $S$ is a sporadic simple group or the Tits group, then $S$ has two nontrivial irreducible
characters of coprime degrees which both extend to ${\rm{Aut}}(S).$

$(iii)$ If $S$ is a simple group of Lie type then the Steinberg character $St_S$ of $S$ of degree
$|S|_p$ extends to ${\rm{Aut}}(S).$
\end{lemma}

\begin{lemma}\emph{(\cite[Lemma $5$]{Bia}).}\label{lem6} Let $N$ be
a minimal normal subgroup of $G$ so that $N\cong S^k,$ where $S$ is a nonabelian simple group. If
$\theta\in {\rm{Irr}}(S)$ extends to ${\rm{Aut}}(S),$ then $\theta^k\in {\rm{Irr}}(N)$ extends to
$G.$
\end{lemma}

The following result due to R. Higgs will be used to verify Step $3.$ The statement given below can
be found in \cite[Theorem $2.3$]{Moreto}.

\begin{lemma}\emph{(\cite[Theorem \textrm{B}]{Higgs}).}\label{lem12}
Let $N$ be a normal subgroup of a group G and let $\theta\in {\rm{Irr}}(N)$ be G-invariant. Assume
that $\chi(1)/\theta(1)$ is a power of a fixed prime p for every $\chi\in {\rm{Irr}}(G|\theta).$
Then $G/N$ is solvable.
\end{lemma}

The following lemma will be used to verify Step $4.$

\begin{lemma}\emph{(\cite[Lemma $6$]{Hupp}).}\label{lem7} Suppose that $M\unlhd
G'=G''$ and that for any $\lambda\in \textrm{Irr}(M)$ with $\lambda(1)=1,$ $\lambda^g=\lambda$ for
all $g\in G'.$ Then $M'=[M,G']$ and $|M/M'|$ divides the order of the Schur multiplier of $G'/M.$
\end{lemma}

\section{The simple Ree groups}
The Ree group ${}^2F_4(q^2),$ where $q^2=2^{2m+1}$ with $m\geq 0,$ is an exceptional group of Lie
type of rank $2$ discovered by Ree in \cite{Ree}. The order of this group is
$$q^{24}(q^{12}+1)(q^8-1)(q^6+1)(q^2-1).$$ This group is nonabelian simple unless $m=0.$ In this
case, the group ${}^2F_4(2)'$ is simple and is called the Tits group. In his preprint, Huppert
verified the conjecture for the Tits group so that we can assume $m\geq 1.$ The character table of
this family of simple groups is available in \cite{Chevie} and is reproduced in Table \ref{Tab2}. The maximal subgroups of ${}^2F_4(q^2)$
were determined by G. Malle in \cite{Malle91}. In Table \ref{Tab1}, we list the maximal subgroups
of ${}^2F_4(q^2)$ together with their indices. We denote by $\Phi_n:=\Phi_n(q),$ the cyclotomic
polynomial in variable $q.$ We have
$$\Phi_1\Phi_2=q^2-1,\Phi_4=q^2+1,\Phi_8=q^4+1,
\Phi_{12}=q^4-q^2+1,\Phi_{24}=q^8-q^4+1.$$ In Table \ref{Tab1}, we use the following notation.
$$u_1:=q^2-\sqrt{2}q+1, u_2:=q^2+\sqrt{2}q+1,$$ $$w_1:=q^4-\sqrt{2}q^3+q^2-\sqrt{2}q+1,
w_2:=q^4+\sqrt{2}q^3+q^2+\sqrt{2}q+1.$$ Then $u_1u_2=\Phi_{8}$ and $w_1w_2=\Phi_{24}.$ Recall that
if $n$ is a positive integer and $p$ is a prime then $n_p$ and $n_{p'}$ are the largest $p$-part
and $p'$-part of $n,$ respectively. That is $n=n_pn_{p'},$ where $(n_p,n_{p'})=1$ and $n_p$ is a
$p$-power.

Let $\ell_i\neq 3,i=1,2,3,$ be prime divisors of $w_1,w_2$ and $\Phi_{12},$ respectively. In the
next lemma, we collect some properties of the character degree set of the simple Ree group
${}^2F_4(q^2),q^2=2^{2m+1},$ where $m\geq 1.$

\begin{lemma}\label{lem8} Let $H$ be the simple Ree group ${}^2F_4(q^2),q^2=2^{2m+1}, m\geq 1$ and
let $a$ be a nontrivial character degree of $H.$ Then the following hold.

$(i)$ If  $a\neq q^{24}$ and $(\ell_1\ell_2,a)=1,$ then $a$ is one of the following degrees:
$$q\sqrt{2}\Phi_1\Phi_2\Phi_4^2\Phi_{12}/2,q^4\Phi_1^2\Phi_2^2\Phi_4^2\Phi_8^2/3,q^{13}\sqrt{2}\Phi_1\Phi_2\Phi_4^2\Phi_{12}/2.$$

$(ii)$ If  $a\neq q^{24}$ and $(\ell_3,a)=1,$ then $a$ is one of the following degrees:
$$\Phi_1\Phi_2\Phi_8^2\Phi_{24}, q^4\Phi_1^2\Phi_2^2\Phi_4^2\Phi_{24}/6,q^4\Phi_8^2\Phi_{24}/2,$$ $$q^2\Phi_1^2\Phi_2^2\Phi_8^2\Phi_{24},
q^6\Phi_1\Phi_2\Phi_8^2\Phi_{24},\Phi_1^2\Phi_2^2\Phi_4^2\Phi_8^2\Phi_{24},q^4\Phi_1^2\Phi_2^2\Phi_4^2\Phi_8^2/3.$$

$(iii)$ We have that $(2\Phi_1\Phi_2\Phi_4,a)>1.$

$(iv)$ If  $(\ell_1\ell_2\ell_3,a)=1$ then $a\in \{q^{24},q^4\Phi_1^2\Phi_2^2\Phi_4^2\Phi_8^2/3\}.$

$(v)$ We have that $q^4\Phi_1^2\Phi_2^2\Phi_4^2\Phi_8^2/3$ is an isolated degree of $H.$

$(vi)$ If $x,y\in {\rm{cd}}(H)-\{1,q^{24}\},$ then $(x,y)>1.$

$(vii)$ $H$ has no consecutive degrees.

$(viii)$ If $a\neq q^{24},$ then $a_2\leq q^{13}\sqrt{2}/2=2^{13m+6}.$

$(ix)$ If $a,b\in {\rm{cd}}(H)$ such that $b=z a,$ where $z>1$ is  odd, then $z\geq q^2-1.$

$(x)$ The smallest nontrivial character degree of $H$ is
$q\sqrt{2}\Phi_1\Phi_2\Phi_4^2\Phi_{12}/2.$
\end{lemma}
\begin{proof} Statements $(i)-(v)$ and $(viii)-(x)$ are obvious by checking Table \ref{Tab2}.
For $(vi),$ if $\ell_3$ divides both $x$ and $y,$ then we are done. Hence we assume that
$(\ell_3,x)=1$  or $(\ell_3,y)=1.$ Without loss of generality, assume $(\ell_3,x)=1.$ Then $x$ is
one of the degrees appearing in $(ii).$ It follows that either $\ell_1\ell_2$ divides $x$ or
$\ell_1\ell_2$ is  prime to $x.$ Assume first that $\ell_1\ell_2\mid x.$ If $(\ell_1\ell_2,y)>1$
then we are done. So assume $(\ell_1\ell_2,y)=1$ so that $y$ is one of the degrees in $(i).$ In
this case, we can see that $2\Phi_1\Phi_2\Phi_4$ divides $y$ and so we have that $(x,y)$ is
divisible by $(x,2\Phi_1\Phi_2\Phi_4).$ Applying $(iii)$ we obtain $(x,2\Phi_1\Phi_2\Phi_4)>1$ so
that $(x,y)>1.$ Now assume that $(\ell_1\ell_2,x)=1.$ It follows that
$x=q^4\Phi_1^2\Phi_2^2\Phi_4^2\Phi_8^2/3$ and so $2\Phi_1\Phi_2\Phi_4$ divides $x.$ Using the same
argument, we have that $(x,y)$ is divisible by $(2\Phi_1\Phi_2\Phi_4,y)$ which is nontrivial by
$(iii)$ so that $(x,y)>1.$ This proves $(vi).$ Next we will show that $H$ has no two consecutive
degrees. By way of contradiction, assume that there exist $x,y\in {\rm{cd}}(H)$ such that $x=y+1.$
Since $H$ is nonabelian simple, it has no character of degree $2$ so that we can assume $y>1$ and
then $x>y>1.$ As $x=y+1,$ we deduce that $(x,y)=1$ and since $x>y>1,$ by $(vi),$ we have $x=q^{24}$
or $y=q^{24}.$ It follows that $q^{24}-1\in {\rm{cd}}(H)$ or $q^{24}+1\in {\rm{cd}}(H).$ However we
can check that $H$ has no such degrees. This contradiction proves $(vii).$
\end{proof}

\begin{lemma}\label{lem9} Let $H$ be the Ree group ${}^2F_4(q^2),q^2=2^{2m+1}, m\geq
1.$ If $K$ is a maximal subgroup of $H$ such that the index $|H:K|$ divides some character degree
$\chi(1)$ of $H,$ then one of the following cases holds:

$(i)$ $K\cong P_a,$ $|H:P_a|=\Phi_4\Phi_8^2\Phi_{12}\Phi_{24}$ and
$$\chi(1)/|H:P_a|\in \{1,q^2-1,q^2,q^2+1\}.$$

$(ii)$ $K\cong P_b,$ $|H:P_b|=\Phi_4^2\Phi_8\Phi_{12}\Phi_{24}$ and
$$\chi(1)/|H:P_b|\in \{q\sqrt{2}/2(q^2-1),u_1(q^2-1),u_2(q^2-1),(q^2-1)^2,q^4,q^4+1\}.$$
\end{lemma}

\begin{proof} If $K$ is one of the parabolic subgroups $P_a$ or $P_b,$ then the result is
obvious. For the remaining maximal subgroups of $H$ except ${}^2F_4(q_0^2),$ we can see that the
$2$-part of the indices are larger than $q^{13}\sqrt{2}/2$ so that these indices cannot divide any
degrees of $H.$ Finally, assume $K\cong {}^2F_4(q_0^2),$ where $q^2=q_0^{2\alpha},\alpha$ prime.
Assume $q_0^2=2^{2l+1}.$ Then $\alpha=\frac{2m+1}{2l+1}$ is odd, so that $\alpha\geq 3$ is an odd
prime. The $2$-part of the index of ${}^2F_4(q_0^2)$ in ${}^2F_4(q^2)$ is $q_0^{24\alpha-24}.$
Moreover as this index is not a $2$-power, it cannot divide the degree of the Steinberg character
of $H$ so that $q_0^{24\alpha-24}\leq q^{13}\sqrt{2}/2<q^{14}=q_0^{14\alpha}.$ It follows that the
$24\alpha-24<14\alpha$ and hence $10\alpha<24.$ Thus $\alpha\leq 2,$ a contradiction.
\end{proof}
The following results are well known, see for example \cite{Hupp}. We note that the inclusion of
the value $(q^2-1)^2$ in Lemma \ref{lem11}$(b)$ causes no difference as it is less than the
smallest index of  the maximal subgroups of the Suzuki groups $Sz(q^2).$
\begin{lemma}\label{lem10} Let $q\geq 8$ be an even prime power. Then the
following hold:

$(a)$ The Schur multiplier of $L_2(q)$ is trivial and ${\rm{cd}}(L_2(q))=\{1,q-1,q,q+1\}.$

$(b)$ If $K$ is a maximal subgroup of $L_2(q)$ whose index  divides some nontrivial character
degree of $L_2(q),$ then $K$ is a Frobenius group of index $q+1.$ Moreover $q+1$ is the smallest
index of maximal subgroups of $L_2(q).$
\end{lemma}

\begin{lemma}\label{lem11} Let $q^2=2^{2m+1},$ where $m\geq 1.$ Then the
following hold.:

$(a)$ The Schur multiplier of $Sz(q^2)$ is trivial when $q^2>8$ while the Schur multiplier of
$Sz(8)$ is elementary abelian of order $4,$ and
$${\rm{cd}}(Sz(q^2))=\{1,q^4,q^4+1,(q^2-1)u_1,(q^2-1)u_2,
q\sqrt{2}(q^2-1)/2\}.$$

$(b)$ If $K$ is a maximal subgroup of $Sz(q^2)$ whose index divides some nontrivial character
degree of $Sz(q^2)$ or $(q^2-1)^2,$ then $K$ is a Frobenius group of index $q^4+1.$ Moreover
$q^4+1$ is the smallest index of  maximal subgroups of $Sz(q^2).$
\end{lemma}

\section{Verifying Huppert's Conjecture for the simple Ree groups}
We are now ready to verify  Huppert's Conjecture for the simple Ree groups.
\subsection{Verifying Step $1$} Show $G'=G''.$ By way of contradiction, suppose
that $G'\neq G''.$ Then there exists a normal subgroup $N\unlhd G$ of $G$ such that $G/N$ is
solvable minimal with respect to being nonabelian.  By Lemma \ref{lem3}, $G/N$ is an $r$-group for
some prime $r$ or $G/N$ is a Frobenius group.

{\bf Case $1.$} $G/N$ is an $r$-group. Then there exists $\psi\in \textrm{Irr}(G/N)$ such that
$\psi(1)=r^b>1.$ By \cite[Theorem $1.1$]{Malle}, we deduce that $\psi(1)=St_H(1)=r^b$ and so $r=2.$
By Thompson's Theorem \cite[Corollary $12.2$]{Isaacs}, $G$ has a nonlinear character $\chi\in
\textrm{Irr}(G)$ such that $\chi(1)$ is odd. As $(\chi(1),|G:N|)=1,$ by Lemma \ref{lem1}$(a),$ we
have $\chi_N\in \textrm{Irr}(N)$ and hence by Gallagher's Theorem, we obtain that $\chi\tau\in
\textrm{Irr}(G)$ for all $\tau\in \textrm{Irr}(G/N).$ Thus $St_H(1)<St_H(1)\chi(1)\in
{\rm{cd}}(G),$ which contradicts Lemma \ref{lem4}.

{\bf Case $2.$} $G/N$ is a Frobenius group with Frobenius kernel $F/N,$ $|F/N|=r^a,$ $1<f=|G:F|\in
{\rm{cd}}(G)$ and $r^a\equiv 1 (\mbox{mod $f$}).$ By Lemma \ref{lem3}$(b)(2),$ if $\chi\in
\textrm{Irr}(G)$ such that $\chi(1)$ is isolated then either $f=\chi(1)$ or $r^a\mid \chi(1)^2.$ By
Lemma \ref{lem4}, the Steinberg character of $H$ is isolated in $H$ and hence by Lemma
\ref{lem3}$(b),$ either $f=q^{24}$ or $r=2.$

Assume first that $f=q^{24}.$ As $r\nmid f,$ $r$ must be odd. Let $\varphi\in \textrm{Irr}(G)$ with
$\varphi(1)=q\sqrt{2}\Phi_1\Phi_2\Phi_4^2\Phi_{12}/2.$  As no proper multiple of $f$ is a character
of $G$ and $\varphi(1)\nmid f,$ we deduce from Lemma \ref{lem3}$(b)(1)$ that $r^a\mid
\varphi(1)^2.$ As $r$ is odd, we obtain $r^a\mid \varphi(1)_{2'}^2=
\Phi_1^2\Phi_2^2\Phi_4^4\Phi_{12}^2.$ We have
$$\Phi_1\Phi_2\Phi_4^2\Phi_{12}=(q^4-1)(q^2+1)(q^4-q^2+1)=(q^4-1)(q^6+1)<q^{10}.$$ As $r^a\mid
\varphi(1)_{2'}^2,$ we deduce that $r^a\leq \varphi(1)_{2'}^2<q^{20}.$ But then as $f\mid r^a-1,$
we obtain  $f=q^{24}\leq r^a-1<q^{20}-1<q^{20},$ which is impossible.

Thus $r=2.$ Then $f\in {\rm{cd}}(G)$ is odd and so $f\neq q^4\Phi_1^2\Phi_2^2\Phi_4^2\Phi_8^2/3,$
which is an isolated degree of $G$ by Lemma \ref{lem8}$(v).$ Hence $r^a\mid
(q^4\Phi_1^2\Phi_2^2\Phi_4^2\Phi_8^2/3)^2$ by Lemma \ref{lem3}$(b)(2).$ It follows that $r^a\mid
q^8$ as $r^a$ is even. As in the previous case, we have that $f$ divides $r^a-1$ and since $r^a\leq
q^8,$ we deduce that $f\leq q^8-1.$ However as $q^8-1<q^{10}\leq
q\sqrt{2}\Phi_1\Phi_2\Phi_4^2\Phi_{12}/2,$ where the latter is
 the smallest nontrivial character degree of $H$ by Lemma \ref{lem8}$(x),$ we see that $f$
cannot be a character degree of $G.$ This contradiction shows that $G'=G''.$

\subsection{Verifying Step $2$ } Let $M\leq G'$
be a normal subgroup of $G$ such that $G'/M$ is a chief factor of $G.$ As $G'$ is perfect, $G'/M$
is nonabelian so that $G'/M\cong S^k$ for some nonabelian simple group $S$ and some integer $k\geq
1.$

$(i)$ Eliminating the alternating groups. Assume that $S=A_n,n\geq 7.$ Let $\theta_i,i=1,2$ be
irreducible characters of $S$ obtained from Lemma \ref{lem5}$(i).$ Then $\theta_1(1)=n(n-3)/2,$
$\theta_2(1)=(n-1)(n-2)/2=\theta_1(1)+1$ and both $\theta_i$ extend to ${\rm{Aut}}(A_n)\cong S_n.$
By Lemma \ref{lem6}, $\theta_i^k\in \textrm{Irr}(G'/M)$ extend to $G/M,$ hence $\theta_i(1)^k\in
{\rm{cd}}(G)$ and $\theta_i(1)^k,i=1,2,$ are coprime.  By Lemma \ref{lem8}$(vi),$ one of the
degrees $\theta_i(1)^k,i=1,2,$ must be $q^{24}.$ However we have that $(n-1,n-2)=1$ and
$(n,n-3)=(n,3)$ so that $\theta_i(1)^k$ can never be a power of $2.$ This shows that $S$ is not an
alternating group of degree at least $7.$

$(ii)$ Eliminating the sporadic simple groups and the Tits group. It follows from \cite[Table $1$]{Bia}) that there exist
two nontrivial irreducible characters $\theta_i,i=1,2,$ such that $\theta_i$ extend to
${\rm{Aut}}(S),$ $\theta_i(1),i=1,2,$ are coprime and $\theta_i(1)$ are not  $2$-power.  Now argue as in case $(i),$ we obtain a contradiction.

$(iii)$ If $S$ is a simple group of Lie type in characteristic $p,$ with $S\neq {}^2F_4(2)',$ then
$k=1$ and $p=2.$ By way of contradiction, assume that $k\geq 2.$ Let $\theta$ be the Steinberg
character of $S.$ Then $\theta(1)=|S|_p$ and $\theta$ extends to ${\rm{Aut}}(S).$ By Lemma
\ref{lem6}, $\theta^k\in \textrm{Irr}(G'/M)$ extends to $G/M,$ hence $\theta(1)^k=|S|_p^k\in
{\rm{cd}}(G).$ Since $q^{24}$ is the unique nontrivial prime power character degree of $G$ by
\cite[Theorem $1.1$]{Malle}, we deduce that $\theta(1)^k=q^{24}.$ In particular, we have $p=2.$
Write $\theta(1)=q_1^s.$ Then $q_1^{sk}=q^{24}.$ Let $\psi=\tau\times \theta\times\cdots\times
\theta\in \textrm{Irr}(S^k),$ where $\tau\in \textrm{Irr}(S)$ with $1<\tau(1)\neq \theta(1).$ Then
$\psi(1)=\tau(1)q_1^{s(k-1)}\nmid q^{24}$ and is nontrivial, so that it must divide some character
degree of $G,$ which is different from $q^{24}.$ By Lemma \ref{lem8}$(viii),$ we have that
$q_1^{s(k-1)}=q^{24(k-1)/k}<q^{14}$ and hence $24(k-1)<14k,$ which implies that $k\leq 2.$
Therefore $k=2.$  Let $C$ be a normal subgroup of $G$ such that $C/M=C_{G/M}(G'/M).$ Then
$G'C/C\cong S^2$ is a unique minimal normal subgroup of $G/C$ so that $G/C$ embeds into
${\rm{Aut}}(S)\wr \Z_2,$ where $\Z_2$ is a cyclic group of order $2.$ Let $B={\rm{Aut}}(S)^2\cap
G/C.$ Then $|G/C:B|=2.$ As above, let $\psi=1\times \theta\in \textrm{Irr}(G'C/C).$ Then $\psi$
extends to $B$ and so $B$ is the inertia group of $\psi$ in $G/C$ so that by Lemma \ref{lem2}$(a),$
$|G/C:B|\psi(1)=2\psi(1)\in {\rm{cd}}(G).$ Hence $2\theta(1)=2q_1^s=2q^{12}\in {\rm{cd}}(G).$
Obviously $1<2q^{12}<q^{24},$ which leads to a contradiction again by using \cite[Theorem
$1.1$]{Malle}. Thus $k=1.$

$(iv)$ If $S$ is a simple group of Lie type in characteristic $2$ and $S\neq {}^2F_4(2)',$ then
$S\cong {}^2F_4(q^2).$ We will prove this by eliminating other possibilities for $S.$ Assume that
$S$ is a simple group of Lie type in characteristic $2$ and $S$ is not the Tits group. We have
shown that $G'/M\cong S$ and $|S|_2=q^{24}=2^{12(2m+1)}.$ Observe that if $\theta\in
\textrm{Irr}(S)$ is extendible to ${\rm{Aut}}(S),$ then $\theta$ extends to $G/C,$ where $C/M=C_{G/M}(G'/M),$ so that
$\theta(1)\in {\rm{cd}}(G).$ In fact, we will choose $\theta$ to be a unipotent character of $S,$
so that by results of Lusztig, $\theta$ is extendible to ${\rm{Aut}}(S)$ apart from some exceptions
(see \cite[Theorem $2.5$]{Malle08}). We refer to \cite[$13.8,13.9$]{Car85} for the classification
of unipotent characters and the notion of symbols. In Table \ref{Tab5}, for each simple group of
Lie type $S$ in characteristic $p,$ we list the $p$-part of  some unipotent character  of $S$ that
is extendible to ${\rm{Aut}}(S).$

$(a)$ Case $S\cong L_n^\epsilon(2^b),$ where $b\geq 1$ and $n\geq 2.$ We have $bn(n-1)=24(2m+1).$
If $n=2$ then $b=12(2m+1)$ so that $S=L_2(q^{24})$ and hence $S$ has a character of degree
$q^{24}+1.$ Obviously this degree does not divide any degree of $G$ since $q^{24}+1\nmid
|{}^2F_4(q^2)|.$ Next if $n=3$ then $b=4(2m+1)$ and so $S=L_3^\epsilon(q^8).$ By
\cite[$(13.8)$]{Car85}, $S$ possesses a unipotent character parametrized by the partition $(1,2)$
of degree $q^8(q^8+\epsilon 1).$ However by checking Table \ref{Tab2}, ${}^2F_4(q^2)$ has no such
degree. If $n=4,$ then $b=2(2m+1)$ so that $S=L_4^\epsilon(q^4).$ In this case, the unipotent
character parametrized by the partition $(2,2)$ has degree $q^8(q^8+1).$ As above, this degree does
not belong to ${\rm{cd}}(G).$ Thus we can assume that $n\geq 5.$ By Table \ref{Tab5}, $S$ possesses
a unipotent character $\chi$ different from the Steinberg character with
$\chi(1)_2=2^{b(n-1)(n-2)/2}.$ By Lemma \ref{lem8}$(viii),$ we have $b(n-1)(n-2)/2<7(2m+1).$
Multiplying both sides by $2n,$ we obtain $bn(n-1)(n-2)=24(2m+1)(n-2)<14n(2m+1),$ and so
$24(n-2)<14n.$ Thus $5n<24$ so that $ n<5,$ which is a contradiction.

$(b)$ Case $S\cong S_{2n}(q_1),$ or $O_{2n+1}(q_1),$ where $q_1=2^b,$ $b\geq 1,$ $n\geq 2$ and
$S\neq S_4(2).$ As $S_{2n}(2^b)\cong O_{2n+1}(2^b),$ we can assume $S=S_{2n}(q_1)$ and $S\neq
S_4(2).$ We have $bn^2=12(2m+1).$ If $n=2$ then $b=3(2m+1)$ and so $S=S_{4}(q^6).$ By
\cite[$(13.8)$]{Car85}, $S$ possesses a unipotent character labeled by the symbol
$\binom{0\:1\:2}{\:\:-\:\:}$ of degree $q^6(q^{6}-1)^2/2.$ However by checking Table \ref{Tab2},
${}^2F_4(q^2)$ has no such character degree. If $n=3$ then $3b=4(2m+1)$ and so $q_1^{3}=q^8.$ By
\cite[$(13.8)$]{Car85}, $S$ possesses a unipotent character labeled by the symbol
$\binom{1\:2}{\:1\:}$ with degree $q_1^3(q_1^2-q_1+1)(q_1^2+q_1+1)=q^8(q_1^2-q_1+1)(q_1^2+q_1+1),$
which leads to a contradiction since $G$ has no degree whose $2$-part is $q^8.$ If $n=4$ then
$4b=3(2m+1)$ and hence $q_1^{4}=q^6.$ By Table \ref{Tab5}, there exists a unipotent character
$\chi$ with $\chi(1)_2=2^{9b-1}.$ As $4b=3(2m+1)$ and $m\geq 1,$ we have that
$9b-1=6(2m+1)+(6m-1)/4= 13m+6+(2m-1)/4>13m+6,$ which contradicts Lemma \ref{lem8}$(viii).$ Hence we
can assume that $n\geq 5.$ By Table \ref{Tab5}, $S$ possesses a nontrivial irreducible character
$\chi$ different from the Steinberg character with $\chi(1)_p=2^{b(n-1)^2-1}.$ Since
$b(n-1)^2-1\geq b(n-1)^2-b=bn(n-2),$ by Lemma \ref{lem8}$(viii),$ we have $bn(n-2)<7(2m+1).$
Multiplying both sides by $n,$ we obtain $bn^2(n-2)=12(n-2)(2m+1)<7n(2m+1)$ so that $5n<24$ and
hence $n<5,$ a contradiction.

$(c)$ Case $S\cong O_{2n}^\epsilon(q_1),$ where $q_1=2^b,$ $b\geq 1,$ and $n\geq 4.$ We have
$bn(n-1)=12(2m+1).$ If $n=4,$ then $b=2m+1$ hence $q_1=q^2.$ Now if $S=O_8^+(q^2),$ then $S$ has a
unipotent character $\chi$ with $\chi(1)_2=q^{2(4^2-3\cdot 4+3)}=q^{14}$ and if $S=O_8^-(q^2),$
then $S$ has a unipotent character $\chi$ with $\chi(1)_2=q^{2(4^2-3\cdot 4+2)}=q^{12},$ by Table
\ref{Tab5}. However $G$ has no such degrees by Table \ref{Tab2} since $q^2\geq 8.$ Observe that
$q^{14}>q^{13}\sqrt{2}/2$ and $q^{13}\sqrt{2}/2=q^{12}$ if and only if $q^2=2.$ Thus we can assume
that $n\geq 5.$ By Table \ref{Tab5}, $S$ possesses a unipotent character $\chi$ different from the
Steinberg character with $\chi(1)_2\geq 2^{b(n-1)(n-2)}.$ By Lemma \ref{lem8}$(viii),$ we have
$b(n-1)(n-2)<7(2m+1).$ Multiplying both sides by $n,$ we obtain
$bn(n-1)(n-2)=12(n-2)(2m+1)<7n(2m+1)$ so that $5n<24$ and hence $n<5,$ which is a  contradiction.

$(d)$ Case $S\cong G_2(q_1),$ where $q_1=2^b,$ $b\geq 1.$ We have $6b=12(2m+1)$ and so $b=2(2m+1).$
Thus $S=G_2(q^4),$ where $q^4>2$ so that $S$ has an irreducible character of degree $q^{24}-1$ by
\cite[Table IV-$2$]{EnoYa}. However this degree divides no degrees of $G$ since
$q^{24}-1=\Phi_1\Phi_2\Phi_3\Phi_4\Phi_6\Phi_8\Phi_{12}\Phi_{24}\nmid |{}^2F_4(q^2)|.$

$(e)$ Case $S\cong {}^2B_2(q_1^2),$ where $q_1^2=2^{2n+1},$ $n\geq 1.$ We have $2(2n+1)=12(2m+1)$
and so $2n+1=6(2m+1),$ which is impossible.

$(f)$ Case $S\cong {}^2G_2(q_1^2),$ where $q_1^2=3^{2n+1},$ $n\geq 1.$ This case cannot occur as
the characteristic of ${}^2G_2(q_1^2)$ is $3$.

$(g)$ Case $S\cong {}^2F_4(q_1^2),$ where $q_1^2=2^{2n+1},$ $n\geq 1.$ We have $12(2n+1)=12(2m+1)$
so that $n=m.$ Thus $S\cong {}^2F_4(q^2).$

$(h)$ For the remaining cases, we can argue as follows. We have $|S|_2=2^{12(2m+1)}.$ Lemma
\ref{lem8}$(viii)$ yields $\chi(1)_2\leq 2^{13m+6},$ where $\chi$ is a unipotent character
different from the Steinberg character listed in Table \ref{Tab5}. Using these two properties, we
will obtain a contradiction. For example, assume $S\cong E_8(q_1),$ where $q_1=2^b,$ $b\geq 1.$ We
have $120b=12(2m+1)$ and so $10b=2m+1.$ By Table \ref{Tab5}, $S$ possesses a unipotent character
$\chi$ with $\chi(1)_2=2^{91b}.$ By Lemma \ref{lem8}$(viii),$ we have $91b<7(2m+1).$ Thus
$91b<7\cdot 10b=70b,$ a contradiction. This completes the proof of Step $2.$
\subsection{Verifying Step $3$} Suppose $\theta\in \textrm{Irr}(M)$ with $\theta(1)=1$ and let $I=I_{G'}(\theta).$ We
need to show that $I=G'.$ By way of contradiction, suppose  $I<G'.$ Write $\theta^I=\sum_{i=1}^s
e_i\mu_i,$ where $\mu_i\in \textrm{Irr}(I).$ As $M\leq I<G'$ and $G'/M\cong H,$ there exists a
subgroup $U$ such that $I\leq U$ and $U/M$ is a maximal subgroup of $G'/M.$ By Lemma
\ref{lem2}$(a),$ we have $|G':I|\mu_i(1)=|G':U||U:I|\mu_i(1)$ divides some character degree of $G.$
Thus the index $|G':U|=|G'/M:U/M|$ of ${}^2F_4(q^2)$ must divide some character degree of
${}^2F_4(q^2)$ and hence by Lemma \ref{lem9}, $U/M\cong P_a$ or $U/M\cong P_b,$ where $P_a$ and
$P_b$ are maximal parabolic subgroups of ${}^2F_4(q^2).$ Note that both unipotent radicals
$[q^{22}]$ and $[q^{20}]$ of the parabolic subgroups $P_a$ and $P_b,$ respectively, are nonabelian
by using the commutator relations (see \cite{HH}). Let $t=|U:I|$ and recall that if $N\unlhd G$ and
$\lambda\in \textrm{Irr}(N),$ then $\textrm{Irr}(G|\lambda)$ denotes the set of all irreducible
constituents of $\lambda^G.$

{\bf Case $1:$} $U/M\cong P_a.$ Recall that $P_a\cong [q^{22}]:(L_2(q^2)\times (q^2-1)).$ Let $L$
and $V$ be subgroups of $U$ containing $M$ such that $L/M\cong [q^{22}]$ and $V/M\cong L_2(q^2),$
and let $W=LV.$ Then $M\unlhd W\unlhd U, L\cap V=M,$  and $W/L\cong V/M\cong L_2(q^2).$ By Lemma
\ref{lem9}$(i),$ $t\mu_i(1)$ divides $q^2\pm 1$ or $q^2.$  We consider the following cases:

{\bf Case $1(a):$} $t\mu_j(1)\mid q^2\pm 1$ for some $j.$ Then $t$ is odd. As $L\unlhd U,$ we
deduce that $I\leq IL\leq U$ so that $t=|U:IL|\cdot |IL:I|.$ Now $|IL:I|=|L:L\cap I|.$ As $|L:L\cap
I|$ divides $|L:M|=q^{20},$ if $L\not\leq I,$ then $|L:L\cap I|>1$ is even so that $t$ is even, a
contradiction. Thus $L\leq I\leq U.$ Assume $W\not\leq I.$ Then $I\lneq WI\leq U$ and
$t=|U:WI|\cdot |WI:I|.$ As $|WI:I|=|W:W\cap I|$ and $|WI:I|>1,$ we deduce that $|W:W\cap I|>1$ and
divides $q^2\pm 1.$ Observe that $W/L\cong \text{L}_2(q^2)$ and since $L\leq W\cap I\lneq W,$ we
deduce that $|W:W\cap I|$ is divisible by some index of a maximal subgroup of $L_2(q^2).$ By Lemma
\ref{lem10}(b), we have $|W:W\cap I|=q^2+1$ and $W\cap I/L$ is isomorphic to the Borel subgroup of
$L_2(q^2),$ in particular $W\cap I/L$ is nonabelian. Furthermore $t=q^2+1$ and hence $t\mu_i(1)\mid
q^2\pm 1,$ for all $i,$ which implies that $\mu_i(1)=1$ for all $i.$ Thus $\theta$ extends to
$\theta_0\in \text{Irr}(I).$ By Gallagher's Theorem, we have $\theta^I=\sum_{\tau\in
\text{Irr}(I/M)}\tau\theta_0$ and so $\tau(1)=1$ for all $\tau\in \text{Irr}(I/M),$ which implies
that $I/M$ is abelian, which is a contradiction as $I/M$ possesses a nonabelian section $W\cap
I/L.$ Thus $W\leq I.$

Let $\lambda$ be an irreducible constituent of $(\mu_j)_L.$ Then $\lambda_M=e\theta$ for some
integer $e.$ Now $\lambda(1)=e$ divides $|L/M|=q^{22}$ by \cite[Corollary $11.29$]{Isaacs}. By
\cite[Lemma $6.8$]{Isaacs}, we have that $\lambda(1)\mid \mu_j(1),$ which yields that $\lambda(1)$
is odd. Thus $e=1$ and so $\lambda$ is an extension of $\theta$ to $L.$ Since $L/M\cong [q^{22}]$
is nonabelian, it possesses a nonlinear irreducible character $\tau$ with even degree. By
Gallagher's Theorem, $\gamma=\tau\lambda\in \textrm{Irr}(L|\theta)$ and
$\gamma(1)=\tau(1)\lambda(1)$ is even. If $\mu_k$ is any irreducible constituent of $\gamma^I$ then
as $L\unlhd I,$ by \cite[Lemma $6.8$]{Isaacs} we have that $\gamma(1)\mid \mu_k(1)$ and since
$t\mu_k(1)$ divides $q^2\pm 1$ or $q^2,$ we deduce that $t\mu_k(1)\mid q^2$ and so as $t$ is odd,
$t=1$ and hence $I=U$ and $\mu_k(1)$ is a $2$-power for any $\mu_k\in \textrm{Irr}(I|\gamma).$ Let
$L\leq J=I_W(\gamma)$ and suppose that $J<W.$ Let $\delta\in \textrm{Irr}(J|\gamma).$ We have
$\delta^W\in \textrm{Irr}(W|\gamma)$ and $\delta^W(1)=|W:J|\delta(1).$ Since $W\unlhd I,$ we deduce
that $|W:J|\delta(1)\mid q^2$ and hence $|W:J|\leq q^2,$ and $|W:J|$ is divisible by the index of
some maximal subgroup of $W/L\cong L_2(q^2),$ so that by Lemma \ref{lem10}(b), we obtain $|W:J|\geq
q^2+1,$ a contradiction. Thus $\gamma$ is $W$-invariant and every irreducible constituent of
$\textrm{Irr}(W|\gamma)$ is a $2$-power. By Lemma \ref{lem12}, we deduce that $W/L\cong L_2(q^2)$
is solvable, which is impossible as $q^2\geq 8.$ Thus this case cannot happen.

{\bf Case $1(b):$} $t\mu_i(1)\mid q^2$ for all $i.$ Then $t\mid q^2$ and all $\mu_i(1)$ are
$2$-powers. We will show that $I/M$ is nonsolvable. If $t=1,$ then $I=U,$ hence $I/M\cong P_a$ is
nonsolvable. Assume $t>1.$ As $W/M\cong [q^{22}]:L_2(q^2)$ is nonsolvable, if $W\leq I,$ then we
are done. So assume $W\not\leq I.$ Let $X=W\cap I.$ Then $X\lneq W.$  Since $|WI:I|=|W:W\cap
I|=|W:X|,$ we have $t=|U:WI|\cdot |W:X|,$ and hence $|W:X|$ is a nontrivial divisor of $q^2.$ If
$L\leq X,$ then $X/L$ is a proper subgroup of $W/L\cong L_2(q^2)$ and hence $|W:X|\geq q^2+1$ by
Lemma \ref{lem10}(b), which is impossible as $1<|W:X|\leq q^2.$  Thus $L\not\leq X.$ Since $L\unlhd
W$ and $X=W\cap I\leq W,$ we deduce that $X\lneq XL\leq W$ and $L\unlhd XL\leq W.$ It follows that
$|W:X|=|W:XL|\cdot |XL:X|$ is a nontrivial divisor of $q^2.$ If $|W:XL|>1,$ then as $L\leq XL\lneq
W,$ by Lemma \ref{lem10}(b), we obtain $|W:XL|\geq q^2+1,$ which is impossible since $|W:XL|\leq
|W:X|\leq q^2.$ Thus $W=XL$ and so $L_2(q^2)\cong W/L\cong  X/X\cap L.$ We have $M\unlhd X\cap
L\unlhd X\unlhd I$ and $X/X\cap L\cong L_2(q^2)$ so that $I/M$ is nonsolvable. Hence $\mu_i(1)$ are
$2$-power for all $i,$ $\theta$ is $I$-invariant and $I/M$ is nonsolvable. Now Lemma \ref{lem12}
will provide a contradiction.


{\bf Case $2:$} $U/M\cong P_b.$ Recall that $P_b\cong [q^{20}]:(Sz(q^2)\times (q^2-1)).$ Let $L$
and $V$ be subgroups of $U$ containing $M$ such that $L/M\cong [q^{20}],$ $V/M\cong Sz(q^2)$ and
let $W=LV.$ It follows that $L\unlhd U, W\unlhd U,V\cap L=M ,$ and $W/L\cong V/M\cong Sz(q^2).$ Let
$\mathcal{B}=\{q^4,q^4+1,(q^2-1)u_1,(q^2-1)u_2,q\sqrt{2}(q^2-1)/2, (q^2-1)^2\}.$ By Lemma
\ref{lem9}$(ii),$ we deduce that for each $i,$ $t\mu_i(1)$ divides one of the members of
$\mathcal{B}.$

{\bf Case $2(a):$} $t$ is odd. As $L\unlhd U,$ we have $I\leq IL\leq U$ so that $t=|U:IL|\cdot
|IL:I|.$ As $|IL:I|=|L:L\cap I|,$ if $L\not\leq I,$ then $|L:L\cap I|>1$ is even so is $t,$ a
contradiction. Thus $L\leq I\leq U.$ Assume that $W\not\leq I.$ Then $I\lneq WI\leq U$ and
$t=|U:WI|\cdot |WI:I|.$ As $|WI:I|=|W:W\cap I|$ and $|WI:I|>1,$ we deduce that $|W:W\cap I|>1$ and
divides one of the members in $\mathcal{B}.$ Observe that $W/L\cong Sz(q^2)$ and since $L\leq W\cap
I\lneq W,$ we deduce that $|W:W\cap I|$ is divisible by some index of a maximal subgroup of
$Sz(q^2).$ By Lemma \ref{lem11}(b), we have $|W:W\cap I|=q^4+1$ and $W\cap I/L$ is isomorphic to
the Borel subgroup of $Sz(q^2),$ in particular $W\cap I/L$ is nonabelian. It follows that $t=q^4+1$
and hence $t\mu_i(1)= q^4+1,$ as $q^4+1$ divides no other members of $\mathcal{B},$ which implies
that $\mu_i(1)=1$ for all $i.$ Now arguing as in the first paragraph of  Case $1(a),$ we obtain
a contradiction. Therefore $W\leq I\leq U.$

Let $\lambda$ be an irreducible constituent of $\theta^L.$ We have that $\lambda(1)=e\theta(1)$ for
some integer $e.$ By \cite[Corollary $11.29$]{Isaacs}, we deduce that $e\mid q^{20}.$ If $e=1$ then
$\theta$ extends to $\lambda\in \textrm{Irr}(L)$ so that as $L/M$ is nonabelian, $L/M$ has a
nontrivial irreducible character $\tau$ of even degree and then by Gallagher's Theorem
$\gamma=\tau\lambda\in \textrm{Irr}(L|\theta)$ with $\gamma(1)$ is even. If $e>1,$ then obviously
$e$ is even and so we choose $\gamma=\lambda\in \textrm{Irr}(L|\theta)$ and $\gamma(1)$ is even. In
both cases, we can choose $\gamma\in \textrm{Irr}(L|\theta)$ such that $\gamma(1)$ is even. Let $J$
be the stabilizer in $W$ of $\gamma.$ Write $\gamma^J=\delta_1+\delta_2+\cdots+\delta_k,$ where
$\delta_i\in \textrm{Irr}(J).$ Since $L\unlhd W\unlhd I,$ the degrees of irreducible constituents
of $\gamma^W$ divide some $\mu_i(1)$ and since $\mu_i(1)$ divides one of the member of
$\mathcal{B},$ we deduce that $\delta_i^W(1)=|W:J|\delta_i(1)$ divides either $q^4$ or
$q\sqrt{2}(q^2-1)/2$ for all $i,$ as $\delta_i(1)$ is even since $\gamma(1)\mid \delta_i(1).$ By
Lemma \ref{lem11}(b), we can deduce that $\gamma$ is $W$-invariant and if $\varphi\in
\textrm{Irr}(W|\gamma)$  then as $W\unlhd I,$ we have that $\varphi(1)$ divides $q^4$ or
$q\sqrt{2}(q^2-1)/2.$

Assume first that $q^2>8.$ By Lemma \ref{lem11}(a), the Schur multiplier of $W/L\cong Sz(q^2)$ is
trivial so that by \cite[Theorem $11.7$]{Isaacs}, $\gamma$ extends to ${\gamma}_0\in
\textrm{Irr}(W).$ Hence by Gallagher's Theorem, $\tau{\gamma}_0$ are all the irreducible constituents
of $\gamma^W,$ where $\tau\in \textrm{Irr}(W/L),$ and thus $\tau(1){\gamma}_0(1)=\gamma(1)\tau(1)$
divides $(q^2-1)q\sqrt{2}/2,$ or $q^4.$ But this is impossible since $W/L\cong Sz(q^2)$ has an
irreducible character of degree $q^4+1,$ which divides none of the degrees above.

Now assume $q^2=8.$ We have $W/L\cong Sz(8),$ $\gamma\in \textrm{Irr}(L)$ is $W$-invariant and all
irreducible constituents of $\gamma^W$ divide $64$ or $14.$ Write
$\gamma^W=f_1\phi_1+f_2\phi_2+\cdots,$ where $\phi_i\in \textrm{Irr}(W).$ Then
$\phi_i(1)=f_i\gamma(1)$ divides $64$ or $14.$ If $f_i=1$ for some $i,$ then $\gamma$ extends to
$\gamma_0\in \textrm{Irr}(W),$ and hence argue as in the previous case to obtain a contradiction.
Thus $f_i>1$ for all $i.$ Also by Lemma \ref{lem2}$(c),$ all $f_i$ are  character degrees of
irreducible projective representations of $W/L\cong Sz(8).$ Thus all $f_i>1$ are character degrees
of irreducible projective representations of $Sz(8)$ with $f_i$ dividing $64$ or $14.$ Using
\cite{atlas}, we have ${\rm{cd}}(Sz(8))=\{1,14,35,64,65,91\},$ and the projective but not ordinary
degrees of $Sz(8)$ are $40,56,64,104.$ It follows that $f_i$ is either $14$ or $64.$ Since
$\gamma^W=f_1\phi_1+\cdots+f_t\phi_t,$ and $\phi_i(1)=f_i\gamma(1),$ we deduce that
$|Sz(8)|=\sum_{i=1}^t f_i^2.$ Let $a$ and $b$ be the numbers of $f_i's$ which equal $14$ and $64,$
respectively. Then $8^2\cdot 5\cdot 7\cdot 13=14^2a+64^2b.$ Obviously both $a$ and $b$ are nonzero.
We have $4\cdot 16\cdot 5\cdot 7\cdot 13=4\cdot 7^2a+4\cdot 16\cdot 8^2b.$ After simplifying, we
obtain $16\cdot 5\cdot 7\cdot 13=7^2a+16\cdot 8^2b,$ and then $b=7b_1$ and $a=16a_1,$ where
$a_1,b_1\geq 1$ are integers. Thus $ 5\cdot 13=7a_1+8^2b_1$ so that $65=7a_1+64b_1.$ As $b_1\geq
1,$ we deduce that $b_1=1$ and then $1=7a_1,$ which is impossible.

{\bf Case $2(b):$} $t>1$ is even. Then $t\mu_i(1)$ divides $q^4$ or $q\sqrt{2}(q^2-1)/2$ for all
$i.$ Let $X=W\cap I\unlhd I.$ Assume first that $L\leq X.$ Then $L\leq I.$ Suppose that $W\not\leq
I.$ Then $I\lneq WI\leq U$ and $t=|U:WI|\cdot |WI:I|.$ As $|WI:I|=|W:X|$ and $|WI:I|>1,$ we deduce
that $|W:X|>1$ and divides $q^4$ or $q\sqrt{2}(q^2-1)/2.$ Hence $|W:X|\leq q^4$ and since $L\leq
X\lneq W,$ we deduce that $|W:X|$ is divisible by the index of some maximal subgroup of $Sz(q^2),$
which is impossible  by Lemma \ref{lem11}(b). Thus $W\leq I.$ But then $t=|U:I|$ divides
$|U:W|=q^2-1,$ which is an odd number, a contradiction. Thus  $L\not\leq X.$ It follows that
$X\lneq XL\leq W.$ We have that $|W:X|=|W:XL|\cdot |XL:X|>1$ and $|W:X|\leq t\leq q^4$ as
$t=|U:WI|\cdot |W:X|.$

Assume  $XL<W.$ We have $L\leq XL<W$ and $X\leq XL<W.$ Since $W/L\cong Sz(q^2),$ we deduce that
$|W:XL|$ is divisible by the index of some maximal subgroup of $Sz(q^2),$ and hence by Lemma
\ref{lem11}(b), $|W:XL|\geq q^4+1,$ which is impossible since $|W:XL|\leq |W:X|\leq q^4.$

Hence $W=XL$ and so $W/L\cong XL/L\cong X/X\cap L.$ Let $L_1=X\cap L.$ Then $M\unlhd L_1\unlhd
X\unlhd I$ and $X/L_1\cong Sz(q^2).$ Let $\lambda\in\textrm{Irr}(L_1|\theta).$ Observe that the
degree of every irreducible constituent of $\lambda^X$ must divide some degree $\mu_i(1)$ so that
it divides $q^4$ or $q\sqrt{2}(q^2-1)/2.$ Now Lemma \ref{lem11}(b) yields that $\lambda$ is
$X$-invariant. Applying the same argument as in the last two paragraphs of Case $2(a)$ for
$\lambda\in\textrm{Irr}(L_1|\theta)$ and $L_1\unlhd X$ with $X/L_1\cong Sz(q^2),$ we obtain a
contradiction.

This finishes the proof of Step $3.$

\begin{table}
 \begin{center}
  \caption{Character degrees of ${}^2F_4(q^2)$}\label{Tab2}
  \begin{tabular}{l|r}
   \hline
   Degree  & Multiplicity \\\hline
   $1$&$1$\\
   $q\sqrt{2}\Phi_1\Phi_2\Phi_4^2\Phi_{12}/2$&$2$\\
   $q^2\Phi_{12}\Phi_{24}$&$1$\\
   $\Phi_1\Phi_2\Phi_{8}^2\Phi_{24}$&$1$\\
   $q^4\cdot u_1^2\cdot w_1\cdot \Phi_1^2\Phi_2^2\Phi_{12}/12$&$1$\\
   $q^4\cdot u_2^2\cdot w_2\cdot \Phi_1^2\Phi_2^2\Phi_{12}/12$&$1$\\
   $q^4\Phi_1^2\Phi_2^2\Phi_4^2\Phi_{24}/6$&$1$\\
   $q^4\cdot w_1\cdot \Phi_1^2\Phi_2^2\Phi_4^2\Phi_{12}/4$&$2$\\
   $q^4\cdot u_1^2\cdot w_2\cdot\Phi_4^2\Phi_{12}/4$&$1$\\
   $q^4\cdot w_2\cdot\Phi_1^2\Phi_2^2\Phi_4^2\Phi_{12}/4$&$2$\\
   $q^4\cdot u_2^2\cdot w_1\cdot\Phi_4^2\Phi_{12}/4$&$1$\\
   $q^4\Phi_1^2\Phi_2^2\Phi_{12}\Phi_{24}/3$&$1$\\
   $q^4\Phi_1^2\Phi_2^2\Phi_4^2\Phi_8^2/3$&$2$\\
   $q^4\Phi_{8}^2\Phi_{24}/2$&$1$\\
   $u_1\cdot\Phi_1\Phi_2\Phi_4^2\Phi_{12}\Phi_{24}$&$q(q+\sqrt{2})/4$\\
   $\Phi_4^2\Phi_{8}\Phi_{12}\Phi_{24}$&$ (q^2-2)/2$\\
   $u_2\cdot\Phi_1\Phi_2\Phi_4^2\Phi_{12}\Phi_{24}$ &  $ (q-\sqrt{2})q/4$\\
   $q^2\Phi_1^2\Phi_2^2\Phi_{8}^2\Phi_{24}$&$ 1$\\
   $\Phi_1\Phi_2\Phi_{8}^2\Phi_{12}\Phi_{24}$&$ (q^2-2)/2$\\
   $q^{10}\Phi_{12}\Phi_{24}$&$ 1$\\
   $\Phi_4\Phi_{8}^2\Phi_{12}\Phi_{24}$&$ (q^2-2)/2$\\
   $q\sqrt{2}\cdot u_1\cdot\Phi_1^2\Phi_2^2\Phi_4^2\Phi_{12}\Phi_{24}/2$& $ (q+\sqrt{2})q/2$\\
   $q^{13}\sqrt{2}\Phi_1\Phi_2\Phi_4^2\Phi_{12}/2$&$ 2$\\
   $q\sqrt{2}\Phi_1\Phi_2\Phi_4^2\Phi_{8}\Phi_{12}\Phi_{24}/2$&$ q^2-{2}$\\
   $q\sqrt{2}\cdot u_2\cdot\Phi_1^2\Phi_2^2\Phi_4^2\Phi_{12}\Phi_{24}/2$& $(q-\sqrt{2})q/2$\\
   $u_1^2\cdot\Phi_1^2\Phi_2^2\Phi_4^2\Phi_{12}\Phi_{24}$&$(q+2\sqrt{2})(q^2-2)q/96$\\
   $w_1\cdot\Phi_1^2\Phi_2^2\Phi_4^2\Phi_{8}^2\Phi_{12}$&$ (q+\sqrt{2})(q^2+1)q/12$\\
   $q^4\cdot u_1\cdot\Phi_1\Phi_2\Phi_4^2\Phi_{12}\Phi_{24}$&$ (q+\sqrt{2})q/4$\\
   $u_1\cdot\Phi_1\Phi_2\Phi_4^2\Phi_{8}\Phi_{12}\Phi_{24}$&$(q-\sqrt{2})q(q+\sqrt{2})^2/8$\\
   $\Phi_1^2\Phi_2^2\Phi_{8}^2\Phi_{12}\Phi_{24}$&$ (q^2-8)(q^2-2)/48$\\
   $q^2\Phi_1\Phi_2\Phi_{8}^2\Phi_{12}\Phi_{24}$&$ (q^2-2)/2$\\
   $\Phi_1^2\Phi_2^2\Phi_4^2\Phi_{8}\Phi_{12}\Phi_{24}$&$(q^2-2)q^2/16$\\
   $q^6\Phi_1\Phi_2\Phi_{8}^2\Phi_{24}$&$ 1$\\
   $\Phi_1\Phi_2\Phi_4\Phi_{8}^2\Phi_{12}\Phi_{24}$&$(q^2-2)q^2/4$\\
   $\Phi_1^2\Phi_2^2\Phi_4^2\Phi_{8}^2\Phi_{24}$&$(q^2-2)(q^2+1)/6$\\
   $q^{24}$&$ 1$\\
   $q^2\Phi_4\Phi_{8}^2\Phi_{12}\Phi_{24}$&$(q^2-2)/2$\\
   $q^4\Phi_4^2\Phi_{8}\Phi_{12}\Phi_{24}$&$(q^2-2)/2$\\
   $\Phi_4^2\Phi_{8}^2\Phi_{12}\Phi_{24}$&$(q^2-8)(q^2-2)/16$\\
   $w_2\cdot\Phi_1^2\Phi_2^2\Phi_4^2\Phi_{8}^2\Phi_{12}$&$(q-\sqrt{2})(q^2+1)q/12$\\
   $q^4\cdot u_2\cdot\Phi_1\Phi_2\Phi_4^2\Phi_{12}\Phi_{24}$&$(q-\sqrt{2})q/4$\\
   $u_2\cdot\Phi_1\Phi_2\Phi_4^2\Phi_{8}\Phi_{12}\Phi_{24}$&$q+\sqrt{2})q(q-\sqrt{2})^2/8$\\
   $u_2^2\cdot\Phi_1^2\Phi_2^2\Phi_4^2\Phi_{12}\Phi_{24}$&$(q-2\sqrt{2})(q^2-2)q/96$\\
    \hline
   \end{tabular}
 \end{center}
\end{table}

\begin{table}
 \begin{center}
  \caption{Some unipotent characters of simple groups of Lie type} \label{Tab5}
  \begin{tabular}{l|l|r}
   \hline
   $S=S(p^b)$  & Symbol &$p$-part of degree\\ \hline
   $L_n^\epsilon(p^b),n\geq 3$ & $(1^{n-2},2)$&$p^{b(n-1)(n-2)/2}$\\
   $S_{2n}(p^b),p=2$&$\binom{0\:1\:2\:\cdots\:n-2\:n-1\:n}{\:\:1\:2\cdots\:n-2}$&$2^{b(n-1)^2-1}$\\

   $S_{2n}(p^b),p>2$ && $p^{b(n-1)^2}$\\
   $O_{2n+1}(p^b),p>2$  &$\binom{0\:1\:2\:\cdots\:n-2\:n-1\:n}{\:\:1\:2\cdots\:n-2}$& $p^{b(n-1)^2}$\\
   $O_{2n}^+(p^b)$&$\binom{0\:1\:2\:\cdots\:n-3\:n-1}{\:1\:2\:3\cdots\:n-2\:n-1}$&$p^{b(n^2-3n+3)}$\\
   $O_{2n}^-(p^b)$&$\binom{0\:1\:2\:\cdots\:n-2\:n-1}{\:\:1\:2\cdots\:n-2}$&$p^{b(n^2-3n+2)}$\\
   ${}^3D_4(p^b)$&$\phi_{1,3}''$&$p^{7b}$\\
   $F_4(p^b)$&$\phi_{9,10}$&$p^{10b}$\\
   ${}^2F_4(q^2)$&${}^2B_2[a],\epsilon$&$\frac{1}{\sqrt{2}}q^{13}$\\
   $E_6(p^b)$&$\phi_{6,25}$&$p^{25b}$\\
   ${}^2E_6(p^b)$&$\phi_{2,16}''$&$p^{25b}$\\
   $E_7(p^b)$&$\phi_{7,46}$&$p^{46b}$\\
   $E_8(p^b)$&$\phi_{8,91}$&$p^{91b}$\\\hline
\end{tabular}
\end{center}
\end{table}

\begin{table}
 \begin{center}
  \caption{The maximal subgroups of  ${}^2F_4(q^2)$}\label{Tab1}
  \begin{tabular}{l|r}
   \hline
   Group  & Index \\\hline
   $P_a=[q^{22}]:(L_2(q^2)\times (q^2-1))$&$(q^{12}+1)(q^6+1)(q^4+1)$\\
   $P_b=[q^{20}]:(Sz(q^2)\times (q^2-1))$&$(q^{12}+1)(q^6+1)(q^2+1)$\\
   $3.U_3(q^2):2$&$q^{18}(q^{12}+1)(q^4+1)(q^2-1)/2$\\
   $(\Z_{q^2+1}\times \Z_{q^2+1}):GL_2(3)$&$q^{24}(q^4+1)^2(q^2-1)^2.\Phi_{12}\Phi_{24}/(3.2^4)$\\
   $(\Z_{q^2-\sqrt{2}q+1}\times \Z_{q^2-\sqrt{2}q+1}):[96]$&$q^{24}(q^4-1)^2.u_2^2.\Phi_{12}\Phi_{24}/(3.2^5)$\\
   $(\Z_{q^2+\sqrt{2}q+1}\times \Z_{q^2+\sqrt{2}q+1}):[96]$&$q^{24}(q^4-1)^2.u_1^2.\Phi_{12}\Phi_{24}/(3.2^5)$\\
   $\Z_{q^4-\sqrt{2}q^3+q^2-\sqrt{2}q+1}:12$&$q^{24}(q^8-1)^2w_2.\Phi_{12}/(3.2^2)$\\
   $\Z_{q^4+\sqrt{2}q^3+q^2+\sqrt{2}q+1}:12$&$q^{24}(q^8-1)^2w_1.\Phi_{12}/(3.2^2)$\\
   $PGU_3(q^2):2$&$q^{18}(q^4+1)(q^2-1)\Phi_{24}/2$\\
   $Sz(q^2)\wr 2$&$q^{16}(q^6+1)(q^2+1)\Phi_{24}/2$\\
   $Sz(q^2):2$&$q^{20}(q^8-1)(q^6+1)\Phi_{24}/2$\\
   ${}^2F_4(q_0^2),q^2=q_0^{2\alpha},\alpha \mbox{ prime}$&$\frac{q_0^{24\alpha}(q_0^{12\alpha}+1)
   (q_0^{8\alpha}-1)(q_0^{6\alpha}+1)(q_0^{2\alpha}-1)}
   {q_0^{24}(q_0^{12}+1)(q_0^8-1)(q_0^6+1)(q_0^2-1)}$\\
  \hline
   \end{tabular}
 \end{center}
\end{table}

\subsection{Verifying Step $4$} Show $M=1.$ We have shown that $G'/M\cong
{}^2F_4(q^2)$ and for any $\theta\in \textrm{Irr}(M),$ if $\theta(1)=1,$ then $\theta$ is
$G'$-invariant so that by Lemma \ref{lem7}, $|M:M'|$ divides the order of the Schur multiplier of
${}^2F_4(q^2).$ As the Schur multiplier of ${}^2F_4(q^2),q^2\geq 8,$ is trivial, we deduce that
$M=M'.$ If $M$ is abelian then we are done. Assume that $M$ is nonabelian. Let $N\leq M$ be a
normal subgroup of $G'$ such that $M/N$ is a chief factor of $G'.$ It follows that $M/N\cong S^k,$
for some nonabelian simple group $S.$ By Lemmas \ref{lem5} and \ref{lem6}, $S$ possesses a
nontrivial irreducible character $\varphi$ such that $\varphi^k\in \textrm{Irr}(M/N)$ which extends
to $G'/N.$ Gallagher's Theorem yields $\varphi(1)^k\tau(1)\in {\rm{cd}}(G'/N)\subseteq
\textrm{cd}(G')$ for any $\tau\in \textrm{Irr}(G'/M)\subseteq \textrm{Irr}(G'/N) .$ Since
${\rm{cd}}(G'/M)={\rm{cd}}(G)$ and $\varphi(1)>1,$ if we choose $\tau$ to be the Steinberg
character of $G'/M,$ then $\varphi(1)^k\tau(1)$ cannot divide any degree of $G,$ a contradiction.
Thus $M=1.$

\subsection{Verifying Step $5$} $G=G'\times C_G(G').$ It follows from Step
$4$ that $G'\cong {}^2F_4(q^2)$ is a nonabelian simple group. Let $C=C_G(G').$ Then $G/C$ is almost
simple with socle ${}^2F_4(q^2).$ Assume that $G'\times C<G.$ Then $G$ induces some outer
automorphism on $G'.$ Note that the only nontrivial outer automorphisms of ${}^2F_4(q^2)$ are field
automorphisms. Let $\sigma$ be a nontrivial outer automorphism of $G'.$ By \cite[Theorem C]{FS},
$\sigma$ does not fix some conjugacy class of $G',$ and so by \cite[Theorem $6.32$]{Isaacs}, the
action of $\sigma$ on the conjugacy classes of $G'$ is permutation isomorphic to the action of
$\sigma$ on $\textrm{Irr}(G'),$ so that $\sigma$ does not fix some nontrivial irreducible character
$\psi\in \textrm{Irr}(G').$ Let $\gamma\in \textrm{Irr}(G)$ be an irreducible constituent of
$\psi^{G}.$ As $\psi$ is not $\sigma$-invariant, we deduce that $\gamma(1)=z \psi(1),$ where $z>1$
and $z\mid |\textrm{Out}(G')|=2m+1.$ We have $\psi(1)>1,$ $z\psi(1)\in {\rm{cd}}(G)$ and
$\psi(1)\in {\rm{cd}}(G')={\rm{cd}}(G).$ Thus $\psi(1)$ and $z\psi(1)$ are in ${\rm{cd}}(G)$ with
$z>1$ being odd, so that by Lemma \ref{lem8}$(ix),$ we have that $z\geq q^2-1.$ But then as $z\mid
2m+1,$ we have $2m+1\geq z\geq 2^{2m+1}-1,$ which is impossible as $m\geq 1.$ Thus $G=G'\times C.$
It follows that $C\cong G/G'$ is abelian. The proof is now complete.

\subsection*{Acknowledgment} The author is grateful to Dr. Thomas Wakefield for many helpful discussions on Huppert's Conjecture.

\end{document}